\documentclass[11pt]{amsart}
\usepackage{amssymb}
\usepackage{epsfig}
\usepackage{mathdots}
\usepackage{hyperref}
\usepackage{fullpage}
\usepackage{color}

\usepackage{array, longtable}
\newcolumntype{C}{>{$}c<{$}}

\theoremstyle{plain}
\newtheorem{theorem}{Theorem}[section]

\newtheorem{lemma}{Lemma}[section]
\newtheorem{proposition}{Proposition}[section]

\newtheorem{example}{Example}[section]
\newtheorem{conjecture}{Conjecture}[section]
\theoremstyle{definition}

\newtheorem{remark}{Remark}[section]

\numberwithin{equation}{section}

\setbox0=\hbox{$+$}
\newdimen\plusheight
\plusheight=\ht0
\def\+{\;\lower\plusheight\hbox{$+$}\;}

\setbox0=\hbox{$-$}
\newdimen\minusheight
\minusheight=\ht0
\def\-{\;\lower\minusheight\hbox{$-$}\;}

\setbox0=\hbox{$\cdots$}
\newdimen\cdotsheight
\cdotsheight=\plusheight
\def\cds{\lower\cdotsheight\hbox{$\cdots$}}

\def\C{\mathbb{C}}

\begin{document}

\title{The Modularity of Z.-W. Sun's Conjectural Formulas for $\frac{1}{\pi}$}

\author{Mark van Hoeij, Wei-Lun Tsai, and Dongxi Ye}

\address{Florida State University, Tallahassee, FL, USA}

\email{hoeij@math.fsu.edu}

\address{Department of Mathematics, University of South Carolina,
 1523 Greene St LeConte College Rm 450
 Columbia, SC 29208}

\email{weilun@mailbox.sc.edu}

\address{
School of Mathematics (Zhuhai), Sun Yat-sen University, Zhuhai 519082, Guangdong,
People's Republic of China}

\email{yedx3@mail.sysu.edu.cn}


\thanks{Mark van Hoeij was supported by NSF grant 2007959. Wei-Lun Tsai was supported by the AMS-Simons Travel Grant. Dongxi Ye was supported by the Guangdong Basic and Applied Basic Research
Foundation (Grant No. 2023A1515010298).}

\begin{abstract}
    In this work, we establish modular parameterizations for two general formulas for $\frac{1}{\pi}$ that subsume conjectural Ramanujan type formulas due to Z.-W. Sun, which have remained open since 2011. As an application of this, in a conceptual way we interpret how Sun's conjectural formulas arise and can be verified, as well as recover other cases that were proved by Cooper, Wan and Zudilin.
\end{abstract}

\maketitle

\allowdisplaybreaks

\section{Introduction}

In one \cite{R14} of the imperative building blocks of his influential mathematics, being facilitated by Jacobi's elliptic integrals and hypergeometric functions Ramanujan ingeniously produces the peculiar yet very intriguing formula
$$
\sum_{n=0}^{\infty}\binom{2n}{n}^{3}\left(n+\frac{5}{42}\right)\frac{1}{2^{12n}}=\frac{8}{21\pi}
$$
and additionally conjectures 17 analogous formulas that are all in form of
$$
\sum_{n=0}^{\infty}a(n)(n+b)=\frac{c}{\pi}
$$
with $a(n),b,c$ being rational numbers by which he introduces ones to a fascinating area of his mathematics called the Ramanujan type series for~$\frac{1}{\pi}$. Following Ramanujan's initiation, as is inspired by his work, relevant researchers also make use of elliptic integrals, modular forms and hypergeometric functions to validate Ramanujan's conjectural formulas and develop new theories to produce uncharted Ramanujan type series. The reader is referred to \cite{BBC} for a brief account of early developments of the topic, and \cite{C17} and the references wherein for recent progress. 

Besides the aforementioned analytic tools, on the path of studying Ramanujan type series a variety of methodologies have also been developed. A notable one that is via supercongruences is due to Z.-W. Sun, who motivated by works of van Hamme \cite{vH} makes use of supercongruences for binomial-like numbers to construct vast novel Ramanujan type series and discover numerous conjectural formulas. In a series of work \cite{S11b, S14, S19, S23}, Sun records a large number of conjectural Ramanujan type series, such as (see \cite[Conj. 41]{S19})
$$
\sum_{n=0}^{\infty}\binom{2n}{n}\left(\sum_{k=0}^{n}\binom{2k}{k}^{2}\binom{2(n-k)}{n-k}\left(\frac{9}{4}\right)^{n-k}\right)\left(n+\frac{1}{12}\right)\left(\frac{1}{100}\right)^{n}=\frac{75}{48\pi}.
$$
While many such formulas have been proved
over the past decade, see Sun's book \cite{S11} for an overview,
others including the one above remained open and challenging for previous methods. 
The main goal of this work is to prove two classes of conjectural Ramanujan type series of Sun that are summarized as follows.

\begin{conjecture}[Z.-W. Sun \cite{S11}]\label{sunconj}
%\begin{enumerate}
    % \item
    % Define $P(t;x,b)$ in $t$ with parameters $x,b$ by
Let
\begin{align}\label{PT}
P(t;x,b)=\sum_{n=0}^{\infty}\binom{2n}{n}\left(\sum_{k=0}^{n}\binom{2k}{k}^{2}\binom{2(n-k)}{n-k}x^{n-k}\right)(n+b)t^{n}.
    \end{align}
    Then
    \begin{align}
        P\left(\frac{1}{100};\frac{9}{4},\frac{1}{12}\right)&=\frac{75}{48\pi},\label{cp1}\\
          P\left(-\frac{1}{192};4,\frac{1}{4}\right)&=\frac{\sqrt{3}}{4\pi},\label{p1}\\
        P\left(\frac{-1}{225};-14,-\frac{224}{17}\right)&=\frac{1800}{17\pi}, \label{cp2}\\
        P\left(\frac{1}{289};18,-\frac{256}{15}\right)&=\frac{2312}{15\pi}, \label{cp3}\\
        P\left(\frac{-1}{576};-32,-\frac{11}{20}\right)&=\frac{9}{2\pi}, \label{cp4}\\
        P\left(\frac{1}{640};36,-\frac{2}{3}\right)&=\frac{5\sqrt{10}}{3\pi}, \label{cp5}\\
        P\left(\frac{-1}{3136};-192,-\frac{67}{20}\right)&=\frac{49}{2\pi}, \label{cp6}\\
        P\left(\frac{1}{3200};196,-\frac{24}{7}\right)&=\frac{125\sqrt{2}}{7\pi}, \label{cp7}\\
        P\left(\frac{-1}{6336};-392,-\frac{32}{5}\right)&=\frac{99}{2\pi}, \label{cp8}\\
        P\left(\frac{1}{6400};396,-\frac{427}{66}\right)&=\frac{500\sqrt{11}}{33\pi}, \label{cp9}\\
        P\left(\frac{-1}{18432};-896,-\frac{7}{34}\right)&=\frac{27\sqrt{2}}{17\pi}, \label{cp10}\\
        P\left(\frac{1}{136^{2}};900,-\frac{5}{24}\right)&=\frac{867}{384\pi}. \label{cp11}
    \end{align}
%\item
% Define $W(t;x,b)$ in $t$ with parameters $x,b$ by
Let
\begin{align}
    \label{WT}
    W(t;x,b)=\sum_{n=0}^{\infty}\left(\sum_{k=0}^{n}\binom{n}{k}\binom{n+k}{k}\binom{2k}{k}\binom{2(n-k)}{n-k}x^{k}\right)(n+b)t^{n}.
\end{align}
Then
\begin{align}
 %   W\left(\frac{1}{6};-\frac{1}{8},\frac{1}{4}\right)&=\frac{\sqrt{72+42\sqrt{3}}}{2\pi}, \label{cw1}\\
    W\left(\frac{-1}{108};-\frac{49}{12},\frac{65}{392}\right)&=\frac{{387\sqrt{3}}}{392\pi}, \label{cw2}\\
    W\left(\frac{1}{112};\frac{63}{16},\frac{23}{168}\right)&=\frac{{59\sqrt{3}}}{54\pi}, \label{cw3}\\
    W\left(\frac{-1}{320};-\frac{405}{64},\frac{257}{1512}\right)&=\frac{{148\sqrt{35}}}{945\pi}, \label{cw4}\\
    W\left(\frac{1}{324};\frac{25}{4},\frac{9}{56}\right)&=\frac{{81\sqrt{35}}}{500\pi}, \label{cw5}\\
    W\left(\frac{-1}{1296};-\frac{625}{9},-\frac{1811}{13000}\right)&=\frac{{12339\sqrt{39}}}{16250\pi}, \label{cw6}\\
    W\left(\frac{1}{1300};\frac{900}{13},-\frac{1343}{9360}\right)&=\frac{{331\sqrt{39}}}{432\pi}, \label{cw7}\\
    W\left(\frac{-1}{5776};-\frac{83521}{361},\frac{2443}{56355}\right)&=\frac{{71839\sqrt{2}}}{58956\pi}, \label{cw8}\\
    W\left(\frac{1}{5780};\frac{1156}{5},\frac{253}{5928}\right)&=\frac{{2227\sqrt{2}}}{1824\pi}.\label{cw9}
\end{align}
%\end{enumerate}
\end{conjecture}

In this work, we shall affirm that

\begin{theorem}\label{main}
    The equalities given in Conjecture~\ref{sunconj} all hold.
\end{theorem}

\begin{remark}
    It is noteworthy that in recent work \cite{S23}, Sun also derives a number of new series for~$\frac{1}{\pi}$ of the form~$W(t;x,b)$ using binomial transformations.
\end{remark}

In addition to proving Theorem~\ref{main} we will establish modular parameterizations for the generic formulas for~$\frac{1}{\pi}$ arising from~\eqref{PT} and~\eqref{WT} (see Proposition~\ref{T1}), so that the problem boils down to CM evaluations of modular forms. For example, we show that there are functions $z_{\pm}(t;x)$ and $h(t;x)$ and a modular function $T(\tau)$, as well as modular forms $Y(\tau)$ and $C(\tau)$ that can be expressed in terms of $T(\tau)$, such that
\begin{align}\label{Pmod}
    P(t;x,b)=\left(\frac{z_{+}'(t;x) h(t;x)t}{2C(\tau_{+})z_{+}(t;x)}\times\left({\frac{Y(\tau_{-})}{Y(\tau_{+})}}\right)^{\frac{1}{2}}+\frac{z_{-}'(t;x) h(t;x)t}{2C(\tau_{-})z_{-}(t;x)}\times\left({\frac{Y(\tau_{+})}{Y(\tau_{-})}}\right)^{\frac{1}{2}}\right)\frac{1}{\pi},
\end{align}
provided that $P(t;x,b)$ converges,
where $b$ is as in Proposition~\ref{T1} and $\tau_{\pm}$ are points in the upper half plane $\mathbb{H}$ listed in Tables 1--3 for which $T(\tau_{\pm})=z_{\pm}(t;x).$

In addition, one shall also see in Section~\ref{cm} that all of the special cases that were also conjectured by Sun and have been proved by others can be recovered by the method used in this work, and more importantly, all of these formulas are exactly Proposition~\ref{T1} specialized to CM points of class number~1 or~2, for whose corresponding $t$ and $x$, their associated $P(t;x,\cdot)$ converges. For the reader's reference, we display these previously proven cases in Theorem~\ref{main2}.

\begin{theorem}[Cooper--Wan--Zudilin \cite{CWZ}]\label{main2}
    Let $P(t;x,b)$ be as in Conjecture~\ref{sunconj}. Then
    \begin{align}
   % {\bf (see~remark)}~ P\left(-\frac{1}{36};-2,-5,2\right)&=\frac{6(7\sqrt{3}-6)}{\pi},\notag\\
        P\left(\frac{1}{100};6,-2\right)&=\frac{50}{3\pi},\label{p2}\\
         P\left(-\frac{1}{192};-8,0\right)&=\frac{3}{2\pi},\label{p3}\\
         P\left(\frac{1}{256};12,-\frac{1}{6}\right)&=\frac{4\sqrt{3}}{3\pi},\label{p4}\\
          P\left(-\frac{1}{1536};-32,\frac{1}{10}\right)&=\frac{3\sqrt{6}}{10\pi},\label{p5}\\
           P\left(\frac{1}{1600};36,\frac{1}{12}\right)&=\frac{75}{96\pi},\label{p6}\\
            P\left(\frac{1}{3136};-60,\frac{5}{24}\right)&=\frac{49\sqrt{3}}{192\pi},\label{p7}\\
             P\left(-\frac{1}{3072};64,\frac{3}{14}\right)&=\frac{3}{7\pi}.\label{p8}
    \end{align}
\end{theorem}

The main strategy of establish modular parameterizations for the Ramanujan type series arising from $P(t;x,b)$ and $W(t;x,b)$, such as~\eqref{Pmod}, is to decompose the generating function of the binomial-like coefficients in terms of local solutions to a Schwarzian differential equation that indeed can be modularly parameterized. To this end, we first 
 display such decompositions in Section~\ref{modpara} followed by the modular parameterizations for the components in these decompositions and prove that these altogether yield general identities interrelating $P(t;x,b)$ or $W(t;x,b)$ with $\frac{1}{\pi}$ via modular forms.  In Section~\ref{cm}, we rigorously determine the CM evaluations of those modular forms that correspond to Sun's formulas, and as a result, validate Conjecture~\ref{sunconj} and re-prove Theorem~\ref{main2} in a uniform way. %At the end, we remark how one may use our general formulas to produce brand new Ramanujan type series with an explicit example.

{\bf Acknowledgment.} The authors thank Professor Zhi-Wei Sun for his useful comments and suggestions. They also thank the referee for their careful reading of the manuscript.

\section{Modular parameterizations}\label{modpara}

In this section, we establish modular parameterizations for the components comprising the decompositions of $P(t;x,0,1)$ and other counterparts obtained in the preceding section, and show how they interplay with~$\frac{1}{\pi}$. As an implication, we relate $P(t;x,b)$ and other counterparts to $\frac{1}{\pi}$ in terms of CM values of modular forms that shall lead us to rigorously verify the equalities in Conjecture~\ref{sunconj}.
We start with a technical but key lemma that facilitates us with finding the corresponding modular parameterizations for $P(t;x,b)$ and $W(t;x,b)$. % can be verified by computer software such as MAPLE.

\begin{lemma}
    \label{decomp}
Let 
$$
P(t;x)=\sum_{n=0}^{\infty}\binom{2n}{n}\left(\sum_{k=0}^{n}\binom{2k}{k}^{2}\binom{2(n-k)}{n-k}x^{n-k}\right)t^{n},
$$
so that $P(t;x,b)$ from Conjecture~\ref{sunconj} equals $t P'(t;x) + b P(t;x)$. Likewise let 
$$
W(t;x)=\sum_{n=0}^{\infty}\left(\sum_{k=0}^{n}\binom{n}{k}\binom{n+k}{k}\binom{2k}{k}\binom{2(n-k)}{n-k}x^{k}\right)t^{n}.
$$
Denote by ${}_{2}F_{1}(a,b;c;z)$ the Gauss hypergeometric function.
%    \begin{enumerate}
%        \item % For $t$ sufficiently close to $0$,
    \iffalse
         \begin{align}
        % &\sum_{n=0}^{\infty}\binom{2n}{n}\left(\sum_{k=0}^{n}\binom{2k}{k}^{2}\binom{2(n-k)}{n-k}x^{n-k}\right)t^{n}\nonumber\\
        P(t;x)
        &= {}_{2}F_{1}\left(\frac{1}{4},\frac{1}{4};1;\frac{128t}{32tx-x+2+\sqrt{x(x-4)(1-64t)}}\right)\\
        & \times  {}_{2}F_{1}\left(\frac{1}{4},\frac{1}{4};1;\frac{128t}{32tx-x+2-\sqrt{x(x-4)(1-64t)}}\right) \times \frac1{\sqrt{1-16tx}}\nonumber.
    \end{align}
\fi
Let
\begin{equation}\label{zpm}
z_{\pm}=z_{\pm}(t;x) = \frac{128t}{32tx-x+2 \pm \sqrt{x(x-4)(1-64t)}},   
\end{equation}
\begin{equation}
    \label{ZF}
    Z_{\pm}(t;x)={}_{2}F_{1}\left(\frac{1}{4},\frac{1}{4};1; z_{\pm}\right).
\end{equation}
\begin{align}\label{fpm}
f_\pm=f_{\pm}(t;x) = \frac{4096t^5 x}{\left(\sqrt{(1-4t)(1-16tx)} \pm \sqrt{(1-4t)^2-16tx}\right)^4}.\end{align}
Then for $t$ sufficiently
close to $0$
\begin{align}\label{pdecomp2}
    P(t;x) = Z_{+}(t;x)\times Z_{-}(t;x) \times \frac1{\sqrt{1-16tx}}
\end{align}
and
    \begin{align}\label{wdecomp}
    % \sum_{n=0}^{\infty}W_{n}(x)t^{n}=
    W(t;x) = 
{}_{2}F_{1}\left(\frac{1}{8},\frac{3}{8};1;f_+ \right) \times
{}_{2}F_{1}\left(\frac{1}{8},\frac{3}{8};1;f_- \right) \times
\frac{1}{ \sqrt{ 1 -4t + 16tx} }. \end{align}

%    \end{enumerate}
\end{lemma}
\begin{proof}
The computation that finds and proves these formulas is very similar to a computation in \cite{saga}, we will give a brief summary. With standard telescoping techniques, one can compute a provably correct recurrence relation for the coefficients of $P(t;x)$. One can convert this recurrence into a differential equation for $P(t;x)$. When a solution~(\ref{pdecomp2}) of that differential equation matches sufficiently many initial conditions, then it must equal $P(t;x)$ for all $t$ sufficiently small.
\end{proof}

Lemma~\ref{decomp} indicates that modular parameterizations for $P(t;x,b)$ and $W(t;x,b)$ can be realized via that for the hypergeometric functions ${}_{2}F_{1}\left(\frac{1}{4},\frac{1}{4};1;\cdot\right)$ and ${}_{2}F_{1}\left(\frac{1}{8},\frac{3}{8};1;\cdot\right)$
and their derivatives.
\begin{remark}
With ${}_{2}F_{1}\left(\frac{1}{4},\frac{1}{4};1; z \right)
={}_{2}F_{1}(\frac{1}{2},\frac{1}{2};1; \frac12-\frac12 \sqrt{1-z})$ we can rewrite (\ref{ZF}) as
\[ Z_{\pm}(t ; x) = {}_{2}F_{1}\left(\frac{1}{2},\frac{1}{2};1;\frac{1}{2} - \frac{1}{2} \frac{ \sqrt{1 - 64t} \pm 16t \sqrt{x^2 - 4x} }{ 1 - 16tx }\right) \]
for $t$ sufficiently small.
The left side of~(\ref{cp1}) is $t P'(t;x) + b P(t;x)$
evaluated at $t,x,b = \frac{1}{100}, \frac94, \frac{1}{12}$. Applying that to~(\ref{pdecomp2})
and replacing ${}_{2}F_{1}\left(\frac{1}{2},\frac{1}{2};1;\cdot \right) = \frac{2}{\pi} K(\sqrt{\cdot})$ this
becomes
\[
\frac{25}{6 \pi^2} \left( E(a_+) K(a_-) + E(a_-) K(a_+) - \frac32 K(a_+)K(a_-) \right)
\]
where $K$ and $E$ are the elliptic $K$ and $E$ functions, and
$a_{\pm}:=({3} \pm \sqrt{-7})/8$. Proving (\ref{cp1}) means
proving that this is $\frac{75}{48\pi}$, which is not trivial. To obtain Theorem~\ref{main} from (\ref{pdecomp2}),(\ref{wdecomp}) we will use modular functions and evaluation at CM points.
\end{remark}

\begin{lemma}
\label{TT}
%\begin{enumerate}
    % \item
    If $T=T(\tau)$ is defined by
    \begin{align}
        \label{ttau}
        T(\tau)=-64\frac{\eta(2\tau)^{24}}{\eta(\tau)^{24}}
    \end{align}
    then $T(\tau)$ is a biholomophic map from $X_{0}(2)$ to $\mathbb{P}^{1}(\mathbb{C})$, and if $|T(\tau)|<1$ then
    \begin{align}
        \label{Fttau}
        {}_{2}F_{1}\left(\frac{1}{4},\frac{1}{4};1;T(\tau)\right)=\frac{1}{T^{\frac{1}{2}}(1-T)^{\frac{1}{4}}}\left(\frac{1}{2\pi i}\frac{dT}{d\tau}\right)^{\frac{1}{2}}.
    \end{align}
     % \item
If $T=T(\tau)$ is defined by
    \begin{align}
        \label{ttau2}
        T(\tau)=\frac{256\eta(\tau)^{24}\eta(2\tau)^{24}}{\left(64\eta(2\tau)^{24}+\eta(\tau)^{24}\right)^{2}}
    \end{align}
    then $T(\tau)$ is a biholomorphic map from $X_{0}(2)+$ to $\mathbb{P}^{1}(\mathbb{C})$, and if $|T(\tau)|<1$ then
    \begin{align}
        \label{Fttau2}
        {}_{2}F_{1}\left(\frac{1}{8},\frac{3}{8};1;T(\tau)\right)=\frac{1}{T^{\frac{1}{2}}(1-T)^{\frac{1}{4}}}\left(\frac{1}{2\pi i}\frac{dT}{d\tau}\right)^{\frac{1}{2}}.
    \end{align}
% \end{enumerate}
\end{lemma}

\begin{proof}
    It is routine to check that the functions $T(\tau)$ are biholomorphic maps defined in their corresponding modular curves.
    The equations~\eqref{Fttau} and~\eqref{Fttau2} basically come from the Fuchsian equations associated to $T(\tau)$. We illustrate this by justifying the case~(1). 

Using the fact that $T(\tau)$ is a biholomorphic map defined in $X_{0}(2)$, one can compute and show that the Schwarzian derivative
$$
R(T)=\frac{3(\tau^{(2)})^{2}-2\tau^{(3)}\tau^{(1)}}{4(\tau^{(1)})^{2}}=\frac{4T^2-5T+4}{16T^2(T-1)^2},
$$
where $\tau^{(n)}=\frac{d^{n}\tau}{dT^{n}}$, and the associated Fuchsian equation is thus given by
\begin{equation}
    \label{fuchs2}
    \frac{d^{2}y}{dT^{2}}+\frac{4T^2-5T+4}{16T^2(T-1)^2}y=0,
\end{equation}
whose non-logarithmic solutions are exactly the scalar multiples of $\left(\frac{1}{(d\tau/dT)}\right)^{\frac{1}{2}}$. Further, one can check that~\eqref{fuchs2} is the projective normal form of 
\begin{equation}
    \label{projfuchs}
    \frac{d^{2}y}{dT^{2}}+\frac{3T-2}{2T(T-1)}\frac{dy}{dT}+\frac{1}{16T(T-1)}y=0,
\end{equation}
    so that the non-logarithmic solutions of~\eqref{projfuchs} are scalar multiples of $\frac{1}{T^{\frac{1}{2}}(1-T)^{\frac{1}{4}}}\left(\frac{1}{(d\tau/dT)}\right)^{\frac{1}{2}}$ which is holomorphic at $T=0$. On the other hand, a local solution of~\eqref{projfuchs} at $T=0$ can be shown to be the hypergeometric series ${}_{2}F_{1}\left(\frac{1}{4},\frac{1}{4};1;T\right)$. Comparing the leading coefficient at $T=0$, one finds that
    $$
    {}_{2}F_{1}\left(\frac{1}{4},\frac{1}{4};1;T\right)=\frac{1}{T^{\frac{1}{2}}(1-T)^{\frac{1}{4}}}\left(\frac{1}{(d\tau/dT)}\right)^{\frac{1}{2}}.
    $$
    Finally, invert by $T=T(\tau)$ to get~\eqref{Fttau}.
\end{proof}

The functions ${}_{2}F_{1}\left(\frac{1}{4},\frac{1}{4};1;T\right)$ and ${}_{2}F_{1}\left(\frac{1}{8},\frac{3}{8};1;T\right)$ can be extended to $T \in \mathbb{C} - [1,\infty)$ via analytic continuation.
When $T=T(\tau)$, as functions in~$\tau$, they are modular forms of weight~$1$. Following the preceding lemma, we next show how these functions interplay with~$\frac{1}{\pi}$.

\begin{lemma}\label{F'F}
%\begin{enumerate}
%    \item 
Let $T=T(\tau)$ be as in~\eqref{ttau}. Then 
    \begin{align}
       & 2\cdot{}_{2}F_{1}\left(\frac{1}{4},\frac{1}{4};1;T(\tau)\right){}_{2}F_{1}'\left(\frac{1}{4},\frac{1}{4};1;T(\tau)\right)
        =\frac{1}{C(\tau)T(\tau)\pi}+\frac{B(\tau)}{T(\tau)}{}_{2}F_{1}\left(\frac{1}{4},\frac{1}{4};1;T(\tau)\right)^{2}\label{FBC}
    \end{align}
    where
    \begin{align}
    Y(\tau)&=\frac{1}{T(\tau)\left(1-T(\tau)\right)^{\frac{1}{2}}}\frac{1}{2\pi i}\frac{dT}{d\tau}\label{YBC1},\\
    B(\tau)&=\frac{1}{(1-T(\tau))^{\frac{1}{2}}Y(\tau)^{2}}\left(\frac{1}{2\pi i}\frac{dY}{d\tau}-\frac{Y(\tau)}{2\pi{\rm Im}(\tau)}\right)\label{YBC2},
    \\ 
    C(\tau)&=2{\rm Im}(\tau)(1-T(\tau))^{\frac{1}{2}}\label{YBC3}.
%t(\tau)&=\frac{108\eta(\tau)^{12}\eta(3\tau)^{12}}{(\eta(\tau)^{12}+27\eta(3\tau)^{12})^{2}}
\end{align}
Now let $T=T(\tau)$ be as in~\eqref{ttau2}. Then
 \begin{align*}
    & 2\cdot{}_{2}F_{1}\left(\frac{1}{8},\frac{3}{8};1;T(\tau)\right){}_{2}F_{1}'\left(\frac{1}{8},\frac{3}{8};1;T(\tau)\right)
        =\frac{1}{C(\tau)T(\tau)\pi}+\frac{B(\tau)}{T(\tau)}{}_{2}F_{1}\left(\frac{1}{8},\frac{3}{8};1;T(\tau)\right)^{2}\label{FBC2}
    \end{align*}
where $Y(\tau),B(\tau),$ and $C(\tau)$ are again described in equations \eqref{YBC1}--\eqref{YBC3}.
%t(\tau)&=\frac{108\eta(\tau)^{12}\eta(3\tau)^{12}}{(\eta(\tau)^{12}+27\eta(3\tau)^{12})^{2}}

% \end{enumerate}
\end{lemma}

\begin{proof}
    These follow from the interplay between the modularity of ${}_{2}F_{1}\left(\frac{1}{4},\frac{1}{4};1;T(\tau)\right)$, modular forms, the Maass operator and~$\frac{1}{\pi}$. Similarly, we illustrate with the case~(1). 

Start with the simple relation
\begin{equation}\label{dydtau}
    Y(\tau)\frac{1}{2\pi i}\frac{dY}{d\tau}=Y(\tau)\left(\frac{1}{2\pi i}\frac{dY}{d\tau}-\frac{Y(\tau)}{2\pi{\rm Im}(\tau)}\right)+\frac{Y(\tau)^2}{2\pi{\rm Im}(\tau)}.
\end{equation}
Note by definition and the part (1) of Lemma~\ref{TT} that
$$
Y(\tau)={}_{2}F_{1}\left(\frac{1}{4},\frac{1}{4};1;T(\tau)\right)^{2},
$$
so that
\begin{align}
    Y(\tau)\frac{1}{2\pi i}\frac{dY}{d\tau}&=Y(\tau)\cdot 2\cdot{}_{2}F_{1}\left(\frac{1}{4},\frac{1}{4};1;T(\tau)\right){}_{2}F_{1}'\left(\frac{1}{4},\frac{1}{4};1;T(\tau)\right)\frac{1}{2\pi i}\frac{dT}{d\tau}\nonumber\\
    &=2\cdot{}_{2}F_{1}\left(\frac{1}{4},\frac{1}{4};1;T(\tau)\right){}_{2}F_{1}'\left(\frac{1}{4},\frac{1}{4};1;T(\tau)\right)T(\tau)(1-T(\tau))^{\frac{1}{2}}Y(\tau)^{2}\label{yyy}.
\end{align}
Then substituting~\eqref{yyy} into~\eqref{dydtau}, replacing the leading $Y(\tau)$ in the first term on the right hand side of~\eqref{dydtau} with ${}_{2}F_{1}\left(\frac{1}{4},\frac{1}{4};1;T(\tau)\right)^{2}$ and dividing both sides by $T(\tau)(1-T(\tau))^{\frac{1}{2}}Y(\tau)^{2}$, one obtains the desired modular identity~\eqref{FBC}.

\end{proof}

We are in a position to state and prove the modular interrelationship between $P(t;x,b)$ or other counterparts, and $\frac{1}{\pi}$ that in some general sense can interpret the modularity of Sun's formulas.

\begin{proposition}\label{T1}
%\textcolor{red}{We can use $z_+$ and $z_-$ from Lemma 2.1, just have to replace $T$ with $t$.}

%\begin{enumerate}
%    \item 

 Let $z_{\pm}(t;x)$ and $T(\tau)$ be respectively defined as in~\eqref{zpm} and~\eqref{ttau}, and let
    \begin{align*}
 %   z_{+}(t;x)&=\frac{128T}{32Tx-x+2+\sqrt{x(x-4)(1-64T)}},\\
  %  z_{-}(t;x)&=\frac{128T}{32Tx-x+2-\sqrt{x(x-4)(1-64T)}},\\
    h(t;x)&=\frac{1}{\sqrt{1-16tx}}.
\end{align*}
For given $t$ and $x$ such that~\eqref{pdecomp2} holds, suppose that $\tau_{\pm}$ are the points in $\mathbb{H}$ such that $T(\tau_{\pm})=z_{\pm}(t;x)$. Then
\begin{align}\label{Key1}
    &P(t;x,b)=\left(\frac{z_{+}'(t;x) h(t;x)t}{2C(\tau_{+})z_{+}(t;x)}\times\left({\frac{Y(\tau_{-})}{Y(\tau_{+})}}\right)^{\frac{1}{2}}+\frac{z_{-}'(t;x) h(t;x)t}{2C(\tau_{-})z_{-}(t;x)}\times\left({\frac{Y(\tau_{+})}{Y(\tau_{-})}}\right)^{\frac{1}{2}}\right)\frac{1}{\pi}
\end{align}
if we let
 \begin{align*}
   b=-\frac{1}{h(t;x)}\left(\frac{B(\tau_{+})z_{+}'(t;x)h(t;x)t}{2z_{+}(t;x)}+\frac{B(\tau_{-})z_{-}'(t;x)h(t;x)t}{2z_{-}(t;x)}+h'(t;x)t\right).
    \end{align*}
Here, $Y(\tau),B(\tau),$ and $C(\tau)$ are as described in equations \eqref{YBC1}--\eqref{YBC3}.

% \item 

 Likewise, let $f_{\pm}(t;x)$ and $T(\tau)$ be respectively defined as in~\eqref{fpm} and~\eqref{ttau2}, and let
    \begin{align*}
 %   z_{+}(t;x)&=\frac{128T}{32Tx-x+2+\sqrt{x(x-4)(1-64T)}},\\
  %  z_{-}(t;x)&=\frac{128T}{32Tx-x+2-\sqrt{x(x-4)(1-64T)}},\\
    h(t;x)&=\frac{1}{\sqrt{1-4t+16tx}}.
\end{align*}
For given $t$ and $x$ such that~\eqref{wdecomp} holds, suppose that $\tau_{\pm}$ are the points in $\mathbb{H}$ such that $T(\tau_{\pm})=f_{\pm}(t;x)$. Then
\begin{align}\label{Key2}
    &W(t;x,b)=\left(\frac{f_{+}'(t;x) h(t;x)t}{2C(\tau_{+})f_{+}(t;x)}\times\left({\frac{Y(\tau_{-})}{Y(\tau_{+})}}\right)^{\frac{1}{2}}+\frac{f_{-}'(t;x) h(t;x)t}{2C(\tau_{-})f_{-}(t;x)}\times\left({\frac{Y(\tau_{+})}{Y(\tau_{-})}}\right)^{\frac{1}{2}}\right)\frac{1}{\pi}
\end{align}
where
 \begin{align*}
    b&=-\frac{1}{h(t;x)}\left(\frac{B(\tau_{+})f_{+}'(t;x)h(t;x)t}{2f_{+}(t;x)}+\frac{B(\tau_{-})f_{-}'(t;x)h(t;x)t}{2f_{-}(t;x)}+h'(t;x)t\right).
    \end{align*}
    Here, $Y(\tau),B(\tau),$ and $C(\tau)$ are as in equations \eqref{YBC1}--\eqref{YBC3}.
% \end{enumerate}

\end{proposition}

\begin{proof}
As lines of derivations are similar, we justify~\eqref{Key1} as an illustration.

Following the notation as defined in Lemma~\ref{decomp}, one can start with the simple relation $P(t;x,b)=tP'(t;x)+bP(t;x)$. By~\eqref{zpm} and part (1) of Lemma~\ref{TT},
\begin{equation}
    \label{pfzz}
    P(t;x)=Z_{+}(t;x)Z_{-}(t;x)h(t;x)=\left(Y(\tau_{+})Y(\tau_{-})\right)^{\frac{1}{2}}h(t;x),
\end{equation}
  and 
one can deduce that
\begin{align}
    tP'(t;x)&= {}_{2}F_{1}'\left(\frac{1}{4},\frac{1}{4};1;z_{+}(t,x)\right)\times Z_{-}(t;x)z_{+}'(t;x) h(t;x)t\nonumber\\
    &\quad +Z_{+}(t;x)\times {}_{2}F_{1}'\left(\frac{1}{4},\frac{1}{4};1;z_{-}(t,x)\right)z_{-}'(t;x) h(t;x)t\nonumber\\
        &\quad+Z_{+}(t;x)\times Z_{-}(t;x) h'(t;x)t\nonumber\\
        &=2\cdot{}_{2}F_{1}\left(\frac{1}{4},\frac{1}{4};1;z_{+}(t,x)\right){}_{2}F_{1}'\left(\frac{1}{4},\frac{1}{4};1;z_{+}(t,x)\right)\times \frac{Z_{-}(t;x)}{2Z_{+}(t;x)}z_{+}'(t;x) h(t;x)t \label{pzzh}\\
    &\quad +\frac{Z_{+}(t;x)}{2Z_{-}(t;x)}\times2\cdot{}_{2}F_{1}\left(\frac{1}{4},\frac{1}{4};1;z_{-}(t,x)\right) {}_{2}F_{1}'\left(\frac{1}{4},\frac{1}{4};1;z_{-}(t,x)\right)z_{-}'(t;x) h(t;x)t\nonumber\\
        &\quad+Z_{+}(t;x)\times Z_{-}(t;x) h'(t;x)t.\nonumber
\end{align}
Note by part (1) of Lemmas~\ref{TT} and~\ref{F'F} that
\begin{equation}\label{zzyy}
    \frac{Z_{\pm}(t;x)}{Z_{\mp}(t;x)}=\left(\frac{Y(\tau_{\pm})}{Y(\tau_{\mp})}\right)^{\frac{1}{2}}
\end{equation}
and
\begin{align}
    &2\cdot{}_{2}F_{1}\left(\frac{1}{4},\frac{1}{4};1;z_{\pm}(t,x)\right){}_{2}F_{1}'\left(\frac{1}{4},\frac{1}{4};1;z_{\pm}(t,x)\right)\nonumber\\
    &=\frac{1}{C(\tau_{\pm})T(\tau_{\pm})\pi}+\frac{B(\tau_{\pm})}{T(\tau_{\pm})}{}_{2}F_{1}\left(\frac{1}{4},\frac{1}{4};1;T(\tau_{\pm})\right)^{2}\nonumber\\
    &=\frac{1}{C(\tau_{\pm})T(\tau_{\pm})\pi}+\frac{B(\tau_{\pm})}{T(\tau_{\pm})}Y(\tau_{\pm}).\label{ctby}
\end{align}
Substituting~\eqref{zzyy} and~\eqref{ctby} into~\eqref{pzzh} and adding up the resulting expression and~\eqref{pfzz}, one obtains 
\begin{align*}
    &P(t;x,b)\\
        &=\left(\frac{z_{+}'(t;x) h(t;x)t}{2C(\tau_{+})z_{+}(t;x)}\times\left({\frac{Y(\tau_{-})}{Y(\tau_{+})}}\right)^{\frac{1}{2}}+\frac{z_{-}'(t;x) h(t;x)t}{2C(\tau_{-})z_{-}(t;x)}\times\left({\frac{Y(\tau_{+})}{Y(\tau_{-})}}\right)^{\frac{1}{2}}\right)\frac{1}{\pi}\nonumber\\
        &\quad+(Y(\tau_{+})Y(\tau_{-}))^{\frac{1}{2}}\left(\frac{aB(\tau_{+})z_{+}'(t;x)h(t;x)t}{2z_{+}(t;x)}+\frac{aB(\tau_{-})z_{-}'(t;x)h(t;x)t}{2z_{-}(t;x)}+ah'(t;x)t+bh(t;x)\right).\nonumber
\end{align*}
Clearly, $h(t;x)\ne0$, so setting 
$$
b=-\frac{1}{h(t;x)}\left(\frac{B(\tau_{+})z_{+}'(t;x)h(t;x)t}{2z_{+}(t;x)}+\frac{B(\tau_{-})z_{-}'(t;x)h(t;x)t}{2z_{-}(t;x)}+h'(t;x)t\right)
$$
one obtains the desired identity.
\end{proof}

\section{CM points and evaluations }\label{cm} To prove Theorem \ref{main} and Theorem,\ref{main2}, we apply Proposition \ref{T1}. Namely, we explicitly compute the right-hand side of the identity \eqref{Key1} (resp. \eqref{Key2}), which reduces the proof to a case-by-case calculation for the coefficient of $1/\pi$ and demonstrates the vanishing of the remaining term. For brevity, we only consider the case \eqref{cp1} when $t=1/100, x =9/4,$ and $b=1/12,$ as the same method applies to all of the cases. All evaluations are detailed in Tables~\ref{table1}--~\ref{table3}.
\vspace{-0.1in}
\subsection{CM points $\tau_{\pm}$}
We first by employing the theory of complex multiplication to calculate the quadratic points on the upper half plane, referred to as CM points $\tau\in\mathbb{H}$. These points must satisfy the condition
\begin{align}\label{Tz}
    T(\tau)=z_{\pm}(t;x)~\mathrm{or}~ f_{\pm}(t;x).
\end{align}
Recall that each CM point $\tau$ of level $N$ corresponds to a positive definite, integral binary quadratic form $[a,b,c]$ with $N|a$ and discriminant $-D=b^2-4ac<0$, following the equation $a\tau^2+b\tau+c=0.$
By observation, $z_{\pm}(t;x)$  and $f_{\pm}(t;x)$ are all algebraic of degree at most $2.$ Since $X_0(2)$ and $X_0(2)+$ have genus zero, $T(\tau)$ generates a number field of degree $H(D)$ or $2H(D),$ bounded by $2,$ where $H(D)$ is the Hurwitz--Kronecker class number of discriminant $-D$\footnote{Note that $-D$ may not be a fundamental discriminant.}. Hence, we provide a list of $D.$ for which $H(D)\leq 2.$
\begin{itemize}
    \item $(H(D)=1):$ $D=3,4,7,8,11,12,16,19,27,28,43,67,163.$
    \item  $(H(D)=2):$ $D=15,20,24,32,35,36,40,48,51,52,60,64,72,75,88,91,99,100,112,115,123,\\ \phantom{}\hspace{1.18in}147,148,187,232,235,267,403,427.$
\end{itemize}
 Hence, by computing all equivalent quadratic forms with $D$ from the provided list, we identify a form whose corresponding CM point $\tau_0\in X_0(2)$ (resp. $X_0(2)+$) satisfies \eqref{Tz}, up to the action of a scaling matrix $\gamma\in {\rm SL}_2(\mathbb{Z}).$

\begin{example} Let $t=1/100, x =9/4.$ Then we have to solve the following equations for $\tau$
\begin{align}\label{TE1}
T(\tau)=z_{\pm}\left(\frac{1}{100},\frac{9}{4}\right)=\frac{47\mp45\sqrt{-7}}{128}.
\end{align}
% which the corresponding minimal polynomial is
% \[
% F(X)=64 X^2 + 47 X + 64\in\mathbb{Z}[X].
% \]
Using the algebraic relation with the modular $j$-function
\begin{align}\label{jid}
 j(\tau)=-\frac{(16-64T(\tau))^3}{64T(\tau)},
\end{align}
we find that two ${\rm  SL}_2(\mathbb{Z})$ equivalent forms $[2,1,1]$ and $[2,-1,1]$ provide level 2 CM points of discriminant $-7$
\[
\tau_{\pm}=\frac{\pm1+\sqrt{-7}}{4}\in X_0(2),
\]
such that \eqref{TE1} holds. 
To be more precise, by considering \eqref{TE1} and \eqref{jid}, we solve for the equation involving the singular moduli
\[
  j(\tau)=-3375.
\]
Hence, we consider the short Weierstrass model
$$E/\mathbb{C}: Y^2=X^3+\frac{3j(\tau)}{1728-j(\tau)}X+\frac{2j(\tau)}{1728-j(\tau)}$$
with the $j$-invariant $j(E)=j(\tau)=-3375,$ which implies $E$ is a CM elliptic curve. After a careful computation for each $D$ in the above list, we identify $\tau_{\pm}$ as the desired CM points in Tables 1--3.

\end{example}

\subsection{Evaluations of $B(\tau)$}

For a CM point $\tau_{0}$, suppose that $\tau_{0}=\gamma\cdot\tau_{0}$ for some $\gamma=\begin{pmatrix}a&b\\c&d\end{pmatrix}$ with $ad-bc=m$ square-free and $2|c$ and $\gcd(2,a)=1$. Let $M_{2}'(m)$ denote the set of integral matrices  $\begin{pmatrix}A&B\\C&D\end{pmatrix}$ with $AD-BC=m$ square-free and $2|C$ and $\gcd(2,A)=1$. Then $\Gamma_{0}(2)$ has a well defined group action on $M_{2}'(m)$ by left matrix multiplication, and it is known that the number of right cosets of $\Gamma_{0}(2)\backslash M_{2}'(m)$ is finite, say, $\{\gamma_{j}\}$ with $\gamma_{1}=\gamma$. 

Recall that $B(\tau)=\frac{1}{(1-T(\tau))^{\frac{1}{2}}Y(\tau)^{2}}\left(\frac{1}{2\pi i}\frac{dY}{d\tau}-\frac{Y(\tau)}{2\pi{\rm Im}(\tau)}\right)$.
Let $f(\tau)=\frac{1}{Y(\tau)2\pi i}\frac{dY}{d\tau}-\frac{1}{2\pi {\rm Im}(\tau)}$, so that $B(\tau)=\frac{f(\tau)}{(1-T(\tau))^{\frac{1}{2}}Y(\tau)}$. For a $\gamma_{j}=\begin{pmatrix}A&B\\C&D\end{pmatrix}$ a representative for some right coset of $\Gamma_{0}(2)\backslash M_{2}'(m)$, define $F_{j}(\tau)=f(\tau)-\left.f\right|_{2}\gamma_{j}$. Then one can find that $F_{j}(\tau)$  is meromorphic in $\mathbb{H}$ and notice in particular that, when $\gamma_{j}=\gamma_{1}=\gamma$ and $\tau=\tau_{0}$, 
       \begin{align*}
           F_{1}(\tau_{0})&=f(\tau_{0})-\frac{m}{(c\tau_{0}+d)^{2}}f(\tau_{0})=\frac{(c\tau_{0}+d)^{2}-m}{(c\tau_{0}+d)^{2}}f(\tau_{0}).
       \end{align*}
       So the evaluation of $B(\tau_{0})$ amounts to evaluating $\frac{(c\tau_{0}+d)^{2}}{((c\tau_{0}+d)^{2}-n)(1-T(\tau_{0}))^{\frac{1}{2}}}\frac{F_{1}(\tau_{0})}{Y(\tau_{0})}$.

         Now consider the modular polynomial
   $$
   \Phi(X)=\prod_{j=1}^{n}\left(X-F_{j}(\tau)\right).
   $$
It is clear by symmetry that the coefficient of $X^{\ell}$ for $0\leq\ell\leq n$ in $\Phi(X)$ is a meromorphic modular form of weight~$2(n-\ell)$ for $\Gamma_{0}(2)$ and thus, can be written as $Y(\tau)^{n-\ell}Q_{\ell}(T(\tau))$ for some rational function $Q_{\ell}$ in $T(\tau)$. Therefore, $\frac{F_{j}(\tau)}{Y(\tau)}$ is a root of 
$$
\Psi(X)=\sum_{\ell=1}^{n}Q_{\ell}(T(\tau))X^{\ell}.
$$

\begin{example}
    Take $\tau_{0}=\frac{1+\sqrt{-7}}{4}$, so that $\gamma=\begin{pmatrix}
        1&-1\\2&0
    \end{pmatrix}$ of determinant~$2$. Then 
    $$
    \Gamma_{0}(2)\backslash M_{2}'(2)=\left\{\begin{pmatrix}
        1&-1\\2&0
    \end{pmatrix},\begin{pmatrix}
        1&0\\0&2
    \end{pmatrix}\right\},
    $$
and one can find that
\begin{align*}
    \frac{F_{1}(\tau)}{Y(\tau)}+\frac{F_{2}(\tau)}{Y(\tau)}&=0,\\
    \frac{F_{1}(\tau)}{Y(\tau)}\times\frac{F_{2}(\tau)}{Y(\tau)}&=\frac{1}{4}T(\tau).
\end{align*}
As an implication, one can tell that $\frac{F_{1}(\tau_{0})}{Y(\tau_{0})}$ is a root of 
$$
\Psi(X)=X^{2}+\frac{1}{4}T(\tau_{0})=X^{2}+\frac{47-45\sqrt{-7}}{512}
$$
and find that
$$
\frac{F_{1}(\tau_{0})}{Y(\tau_{0})}=-\sqrt{\frac{-47+45\sqrt{-7}}{512}}.
$$
As a result of this, one obtains
$$
B(\tau_{0})=\frac{2\tau_{0}^{2}}{(2\tau_{0}^{2}-1)(1-T(\tau_{0}))^{\frac{1}{2}}}\frac{F_{1}(\tau_{0})}{Y(\tau_{0})}=-\frac{1}{4}-\frac{\sqrt{-7}}{84}
$$
 displayed as in Table 1.
\end{example}

\subsection{Evaluations of $\frac{Y(\tau_{+})}{Y(\tau_{-})}$} 

Let $y(\tau)=\frac{1}{2\pi i}\frac{dt}{d\tau}$. So it suffices to evaluate $\frac{y(\tau_{+})}{y(\tau_{-})}$.
Notice that $\tau_{+}$ and $\tau_{-}$ can be algebraically related by some integral matrix $\gamma=\begin{pmatrix}
    a&b\\c&d
\end{pmatrix}$ via $\tau_{+}=\gamma\cdot\tau_{-}$, so the evaluation of $\frac{y(\tau_{+})}{y(\tau_{-})}$ amounts to computing the modular function $\xi_{\gamma}(\tau)=\frac{\left.y\right|_{2}\gamma}{y(\tau)}$ at the CM point $\tau_{-}$. Suppose that $\gamma\in M_{2}(m)$, where $M_{2}(m)$ denotes the set of integral matrices $\begin{pmatrix}
    A&B\\C&D
\end{pmatrix}$ with $2|C$. So it is clear that $\xi_{\gamma}(\tau)$ only depends on residue classes of $\Gamma_{0}(2)\backslash M_{2}(m)$, and thus by symmetry one can tell that the coefficients of the polynomial
$$
\Xi(X;\tau)=\prod_{\gamma\in\Gamma_{0}(2)\backslash M_{2}(m)}\left(X-\xi_{\gamma}(\tau)\right)
$$
are modular functions for $\Gamma_{0}(2)$ and thus, rational functions $R_{j}(y)$ in $T(\tau)$, that is to say,
$$
\Xi(X;\tau)=X^{n}+R_{n-1}(T(\tau))X^{n-1}+\cdots+R_{0}(T(\tau)).
$$
Therefore, $\xi_{\gamma}(\tau_{-})$ must be a root of $\Xi(x;\tau_{-})$, and this provides a rigorous evaluation to $\frac{y(\tau_{+})}{y(\tau_{-})}$.

\begin{example}   Take $\tau_{\pm}=\frac{\pm1+\sqrt{-7}}{4}$ and $\gamma=\begin{pmatrix}
        0&-1\\2&0
    \end{pmatrix}\in M_2(2)$ so that $\tau_+=\gamma\cdot\tau_-.$ 
 Note that 
    $$
    \Gamma_{0}(2)\backslash M_{2}(2)=\left\{\gamma_{1}=\begin{pmatrix}
        1&0\\0&2
    \end{pmatrix},\gamma_{2}=\begin{pmatrix}
        1&1\\0&2
    \end{pmatrix}, \gamma_{3}=\begin{pmatrix}
        2&0\\0&1
    \end{pmatrix},\gamma_{4}=\begin{pmatrix}
        0&-1\\2&0
    \end{pmatrix},\gamma_{5}=\begin{pmatrix}
        0&-1\\2&1
    \end{pmatrix}\right\},
    $$
and one can find that
\begin{align*}
    \Xi(X;\tau)&=X^{5}+(128T(\tau)-48)X^{4}-\frac{(32767T(\tau)^4+4088T(\tau)^3+256T(\tau)-256)}{(256T(\tau)^4(T(\tau)-1))}X^{3}\\
    &\quad+\frac{(128T(\tau)^5+976T(\tau)^4-384T(\tau)^3-32767T(\tau)^2+45064T(\tau)-12288)}{(256T(\tau)^4(-1+T(\tau)))}X^{2}\\
    &\quad-\frac{(12T(\tau)^4+93T(\tau)^3-2072T(\tau)^2+1792T(\tau)+256)}{(16T(\tau)^5(T(\tau)-1)^2)}X-\frac{(T(\tau)+8)(8T(\tau)+1)}{(16T(\tau)^4(T(\tau)-1)^2)}
\end{align*}
Then $\frac{y|_{2}\gamma_{4}}{y}(\tau_{-})$ is a root of
\begin{align*}
   \Xi(X;\tau_{-})&=\frac{1}{4294967296} (67\sqrt{-7}+8192X-161)(2115\sqrt{-7}-8192X+5983)(181\sqrt{-7}+4X-7)\\
   &\quad\times(\sqrt{-7}-4X-3)(\sqrt{-7}+4X+3).
\end{align*}
So making use of the fact $T(\tau_{-})=\frac{47-45\sqrt{-7}}{128}$, and comparing the values numerically, one can find that
$$
\frac{y|_{2}\gamma_{4}}{y}(\tau_{-})=\frac{5983+2115\sqrt{-7}}{8192}.
$$
Subsequently, this yields the desired result
$$
\frac{Y(\tau_{+})}{Y(\tau_{-})}=\frac{31-3\sqrt{-7}}{32}.
$$
    
\end{example}

\vspace{-0.199in}

\begingroup
\setlength{\tabcolsep}{2.5pt} 
\renewcommand{\arraystretch}{1.5}
\begin{center}
\footnotesize
\begin{table}[!ht]
\begin{tabular}{|c|c|c|c|c|c|c|c|}
 \hline
%\multicolumn{3}{c}{Defective $u_n(\alpha, \beta)$ for $n \in \{ 3, 4, 6 \}$} \\ \hline
%\multicolumn{3}{c}{} \\ \hline
Case & $(t;x)$ & $h(t;x)$ & $z_\pm(t;x)$ & $\tau_\pm$&$B(\tau_{\pm})$  & $C(\tau_{\pm})$ & $\frac{Y(\tau_{\pm})}{Y(\tau_{\mp})}$   \\ \hline \hline
\eqref{cp1} &$( \frac{1}{100},\frac{9}{4})$ & $\frac{5}{4}$ & $\frac{47\mp45\sqrt{-7}}{128}$ & $\frac{\pm1+ \sqrt{-7}}{4}$   & $\frac{-1}{4}\mp\frac{\sqrt{-7}}{84}$& $\frac{15\sqrt{7}\pm 21\sqrt{-1}}{32}$ & $\frac{31\mp3\sqrt{-7}}{32}$ \\ \hline
\eqref{p1} &$( -\frac{1}{192},4)$ & $\frac{\sqrt{3}}{2}$ & $\frac{1}{4},\frac{1}{4}$ & $\frac{\pm1+ \sqrt{-3}}{2}$   & $-\frac{1}{6},-\frac{1}{6}$& $\frac{3}{2},\frac{3}{2}$ & $1,1$ \\ \hline
\eqref{cp2}&$(  \frac{-1}{225},-14)$ & $15$& $\hspace{-1pt}-32(45\hspace{-2pt}\mp\hspace{-2pt}17\sqrt{7})^2\hspace{-1pt}$ & $\frac{\sqrt{-7}}{2},\frac{\sqrt{-7}}{14}$  & $\frac{-16}{57}\pm\frac{8\sqrt{7}}{133}$ & $-672+255\sqrt{7}, \frac{672+255\sqrt{7}}{7}$ 
 & $\frac{127+48\sqrt{7}}{7}, 889-336\sqrt{7}$\\ \hline
\eqref{cp3}&$(  \frac{1}{289},18)$ & $17$ & $\hspace{-1pt}-32(45\hspace{-2pt}\pm\hspace{-2pt}17\sqrt{7})^2\hspace{-1pt}$& $\frac{\sqrt{-7}}{14}, \frac{\sqrt{-7}}{2}$  & $\frac{-16}{57}\mp\frac{8\sqrt{7}}{133}$ & $ \frac{672+255\sqrt{7}}{7}, -672+255\sqrt{7}$  &$889-336\sqrt{7}, \frac{127+48\sqrt{7}}{7}$   \\ \hline
 \eqref{cp4}&$(  \frac{-1}{576},-32)$& $3$ & $-(9\mp4\sqrt{5})^2$& $\frac{\sqrt{-10}}{2},\frac{\sqrt{-10}}{10}$  & $\frac{-1}{4}\pm\frac{\sqrt{5}}{15}$ & $30-12\sqrt{5}, \frac{30+12\sqrt{5}}{5}$ & $\frac{9+4\sqrt{5}}{5}, 45-20\sqrt{5}$ \\ \hline
\eqref{cp5}&$(  \frac{1}{640},36)$ & $\sqrt{10}$ & $-(9\pm4\sqrt{5})^2$&$\frac{\sqrt{-10}}{10},\frac{\sqrt{-10}}{2}$  & $\frac{-1}{4}\mp\frac{\sqrt{5}}{15}$ & $\frac{30+12\sqrt{5}}{5}, 30-12\sqrt{5}$ & $45-20\sqrt{5}, \frac{9+4\sqrt{5}}{5}$
 \\ \hline
\eqref{cp6}&$(  \frac{-1}{3136},-192)$& $7$ & $-(49\mp20\sqrt{6})^2$& $\frac{3\sqrt{-2}}{2},\frac{\sqrt{-2}}{6}$ &  $\frac{-1}{4}\pm\frac{\sqrt{6}}{14}$  & $210-84\sqrt{6},\frac{70+28\sqrt{6}}{3}$ & $\frac{49+20\sqrt{6}}{9}, 441-180\sqrt{6}$ \\ \hline
\eqref{cp7}&$(  \frac{1}{3200},196)$& $5\sqrt{2}$ & $-(49\pm20\sqrt{6})^2$& $\frac{\sqrt{-2}}{6},\frac{3\sqrt{-2}}{2}$ &  $\frac{-1}{4}\mp\frac{\sqrt{6}}{14}$ & 
 $\frac{70+28\sqrt{6}}{3}, 210-84\sqrt{6}$ & $441-180\sqrt{6},\frac{49+20\sqrt{6}}{9}$\\ \hline
\eqref{cp8}&$(  \frac{-1}{6336},-392)$& $3\sqrt{22}$ & $-(90\mp70\sqrt{2})^2$& $\frac{\sqrt{-22}}{2}, \frac{\sqrt{-22}}{22}$ & $\frac{-1}{4}\pm\frac{17\sqrt{2}}{132}$ & $-462+330\sqrt{2}, 42+30\sqrt{2}$ & $\frac{99+70\sqrt{2}}{11},1089-770\sqrt{2}$
  \\ \hline
\eqref{cp9}&$(  \frac{1}{6400},396)$& $10$ & $-(90\pm70\sqrt{2})^2$& $\frac{\sqrt{-22}}{22}, \frac{\sqrt{-22}}{2}$  & $\frac{-1}{4}\mp\frac{17\sqrt{2}}{132}$  & $42+30\sqrt{2},-462+330\sqrt{2}$  & $1089-770\sqrt{2}, \frac{99+70\sqrt{2}}{11}$\\ \hline
\eqref{cp10}&$(  \frac{-1}{18432},-896)$& $\frac{3\sqrt{2}}{2}$ & $-\frac{(45\mp17\sqrt{7})^2}{128}$& $\sqrt{-7},\frac{\sqrt{-7}}{7}$  & $\frac{-25}{114}\pm\frac{8\sqrt{7}}{133}$  & $\frac{51\sqrt{14}-105\sqrt{2}}{8},\frac{51\sqrt{14}+105\sqrt{2}}{56}$ & $\frac{8+3\sqrt{7}}{7}, 56-21\sqrt{7}$\\ \hline
\eqref{cp11}&$(  \frac{1}{18496},900)$ & $\frac{17}{8}$& $-\frac{(45\pm17\sqrt{7})^2}{128}$&$\frac{\sqrt{-7}}{7},\sqrt{-7}$   &  $\frac{-25}{114}\mp\frac{8\sqrt{7}}{133}$  & $\frac{51\sqrt{14}+105\sqrt{2}}{56}, \frac{51\sqrt{14}-105\sqrt{2}}{8}$ & $56-21\sqrt{7}, \frac{8+3\sqrt{7}}{7}$ \\ \hline
\end{tabular}
\bigskip
    \caption{CM points $\tau_\pm$ along with corresponding computations for cases \eqref{cp1}--\eqref{cp11}.}
\label{table1}
\textit{}
\end{table}
\normalsize
\end{center}
\endgroup

\begingroup
\setlength{\tabcolsep}{2.pt} 
\renewcommand{\arraystretch}{1.8}
\begin{center}
\footnotesize
\begin{table}[!ht]
\begin{tabular}{|c|c|c|c|c|c|c|c|}
 \hline
%\multicolumn{3}{c}{Defective $u_n(\alpha, \beta)$ for $n \in \{ 3, 4, 6 \}$} \\ \hline
%\multicolumn{3}{c}{} \\ \hline
Case & $(t;x)$ & $h(t;x)$ & $z_\pm(t;x)$ & $\tau_\pm$&$B(\tau_{\pm})$  & $C(\tau_{\pm})$ & $\frac{Y(\tau_{\pm})}{Y(\tau_{\mp})}$   \\ \hline \hline
% {\bf(new?)}&$(  -\frac{1}{36},-2)$ & $3$& $-416\pm240\sqrt{3}$ & $\frac{\sqrt{-3}}{2},\frac{\sqrt{-3}}{6}$  & $\frac{-10\pm2\sqrt{3}}{33}$ & $ -24+15\sqrt{3},8+5\sqrt{3}$ 
%  & $\frac{7+4\sqrt{3}}{3}, 21-12\sqrt{3}$\\ \hline
\eqref{p2}&$(  \frac{1}{100},6)$ & $5$& $-416\mp240\sqrt{3}$ & $\frac{\sqrt{-3}}{6},\frac{\sqrt{-3}}{2}$  & $\frac{-10\mp2\sqrt{3}}{33}$ & $8+5\sqrt{3}, -24+15\sqrt{3}$ 
 & $21-12\sqrt{3}, \frac{7+4\sqrt{3}}{3}$\\ \hline
\eqref{p3}&$(  -\frac{1}{192},-8)$ & $\sqrt{3}$ & $-17\pm12\sqrt{2}$& $\frac{\sqrt{-6}}{2}, \frac{\sqrt{-6}}{6}$  & $\frac{-3\pm\sqrt{2}}{12}$ & $-6+6\sqrt{2}, 2+2\sqrt{2}$  &$\frac{3+2\sqrt{2}}{3}, 9-6\sqrt{2}$   \\ \hline
 \eqref{p4}&$(  \frac{1}{256},12)$& $2$ & $-17\mp12\sqrt{2}$&  $\frac{\sqrt{-6}}{6}, \frac{\sqrt{-6}}{2}$  & $\frac{-3\mp\sqrt{2}}{12}$ & $2+2\sqrt{2},-6+6\sqrt{2}$  &$9-6\sqrt{2},\frac{3+2\sqrt{2}}{3}$   \\ \hline
\eqref{p5}&$(  -\frac{1}{1536},-32)$ & $\frac{\sqrt{6}}{2}$ & $\frac{-26\pm15\sqrt{3}}{16}$& $\sqrt{-3},\frac{\sqrt{-3}}{3}$ &  $\frac{-13\pm4\sqrt{3}}{66}$  & $\frac{15\sqrt{2}-3\sqrt{6}}{4}, \frac{5\sqrt{2}+\sqrt{6}}{4}$ & $\frac{2+\sqrt{3}}{3}, 6-3\sqrt{3}$ \\ \hline
\eqref{p6}&$(  \frac{1}{1600},36)$& $\frac{5}{4}$ & $\frac{-26\mp15\sqrt{3}}{16}$& $\frac{\sqrt{-3}}{3},\sqrt{-3}$ &  $\frac{-13\mp4\sqrt{3}}{66}$  & $\frac{5\sqrt{2}+\sqrt{6}}{4},\frac{15\sqrt{2}-3\sqrt{6}}{4}$ & $6-3\sqrt{3}, \frac{2+\sqrt{3}}{3}$ \\ \hline
\eqref{p7}&$(  \frac{1}{3136},-60)$& $\frac{7}{8}$ & $\frac{47\mp21\sqrt{5}}{128}$& $\frac{-1+\sqrt{-15}}{2},\frac{3+\sqrt{-15}}{6}$ &  $-\frac{3}{22}\pm\frac{4\sqrt{5}}{165}$ & $\frac{15+21\sqrt{5}}{16},\frac{-5+7\sqrt{5}}{16}$
  & $\frac{3+\sqrt{5}}{6},\frac{9-3\sqrt{5}}{2}$\\ \hline
\eqref{p8}&$(  -\frac{1}{3072},64)$& $\frac{\sqrt{3}}{2}$ & $\frac{47\pm21\sqrt{5}}{128}$& $\frac{3+\sqrt{-15}}{6}, \frac{-1+\sqrt{-15}}{2}$ &  $-\frac{3}{22}\mp\frac{4\sqrt{5}}{165}$ & $\frac{-5+7\sqrt{5}}{16}, \frac{15+21\sqrt{5}}{16}$
  & $\frac{9-3\sqrt{5}}{2}, \frac{3+\sqrt{5}}{6}$\\ \hline
\end{tabular}
\bigskip
\caption{CM points $\tau_\pm$ along with corresponding computations for cases \eqref{p1}--\eqref{p8}.}
\label{table2}
\textit{}
\end{table}
\normalsize
\end{center}
\endgroup

\begingroup
\setlength{\tabcolsep}{0.23pt} 
\renewcommand{\arraystretch}{2.3}
\tiny
\begin{center}
\begin{table}[!ht]
\begin{tabular}{|c|c|c|c|c|c|c|c|}
 \hline
%\multicolumn{3}{c}{Defective $u_n(\alpha, \beta)$ for $n \in \{ 3, 4, 6 \}$} \\ \hline
%\multicolumn{3}{c}{} \\ \hline
Case & $(t;x)$ & $h(t;x)$ & $\sqrt{-f_\pm(t;x)}$ & $\tau_\pm$&$B(\tau_{\pm})$  & $C(\tau_{\pm})$&$\frac{Y(\tau_{\pm})}{Y(\tau_{\mp})}$ \\ \hline \hline
%\eqref{cw1} &$( \frac{1}{6},-\frac{1}{8})$ & $??$ & $-\frac{1}{48}, \frac{}{0}?$  &   & & & \\ \hline
\eqref{cw2}&$(  \frac{-1}{108},\frac{-49}{12})$ & $\frac{9\sqrt{133}}{133}$ & $\frac{275\pm112\sqrt{6}}{1083}$ & $\frac{\sqrt{-42}}{2},\frac{\sqrt{-42}}{6}$   & $\frac{-13}{140}\pm\frac{\sqrt{6}}{56}$ &   $\frac{4620\sqrt{2}+280\sqrt{3}}{1083}, \frac{4620\sqrt{2}-280\sqrt{3}}{3249}$ & $\frac{31+10\sqrt{6}}{57},\frac{93-30\sqrt{6}}{19}$\\ \hline
\eqref{cw3}&$(  \frac{1}{112},\frac{63}{16})$ & $\frac{4\sqrt{133}}{57}$&$\frac{275\mp112\sqrt{6}}{1083}$ & $\frac{\sqrt{-42}}{6},\frac{\sqrt{-42}}{2}$ & $\frac{-13}{140}\mp\frac{\sqrt{6}}{56}$  &  $\frac{4620\sqrt{2}-280\sqrt{3}}{3249}, \frac{4620\sqrt{2}+280\sqrt{3}}{1083}$ & $\frac{93-30\sqrt{6}}{19},\frac{31+10\sqrt{6}}{57}$\\ \hline
 \eqref{cw4}&$(  \frac{-1}{320},\frac{-405}{64})$& $\frac{16\sqrt{105}}{189}$ & $\frac{253\pm80\sqrt{10}}{567}$ & $\frac{\sqrt{-70}}{2},\frac{\sqrt{-70}}{10}$  & $\frac{-1543}{16120}\pm\frac{147\sqrt{10}}{8060}$& $\frac{920\sqrt{14}+110\sqrt{140}}{567},\frac{184\sqrt{14}-22\sqrt{140}}{567}$&  $\frac{7+2\sqrt{10}}{15},\frac{35-10\sqrt{10}}{3}$\\ \hline
\eqref{cw5}&$(  \frac{1}{324},\frac{25}{4})$ & $\frac{3\sqrt{105}}{35}$  & $\frac{253\mp80\sqrt{10}}{567}$ & $\frac{\sqrt{-70}}{10},\frac{\sqrt{-70}}{2}$  &$\frac{-1543}{16120}\mp\frac{147\sqrt{10}}{8060}$ & $\frac{184\sqrt{14}-22\sqrt{140}}{567}, \frac{920\sqrt{14}+110\sqrt{140}}{567}$
 & $\frac{35-10\sqrt{10}}{3},\frac{7+2\sqrt{10}}{15}$\\ \hline
\eqref{cw6}&$(  \frac{-1}{1296},\frac{-625}{9})$& $\frac{54\sqrt{217}}{1085}$  & $\frac{1103\pm7800\sqrt{2}}{141267}$  & $\frac{\sqrt{-78}}{2},\frac{\sqrt{-78}}{6}$  &$\frac{-5983\pm2097\sqrt{2}}{83720}$ & $\frac{4420\sqrt{3}}{141267}+\frac{168740\sqrt{6}}{47089}, \frac{-4420\sqrt{3}}{423801}+\frac{168740\sqrt{6}}{141267}$ & $\frac{361}{651}\hspace{-2pt}+\hspace{-2pt}\frac{68\sqrt{2}}{217},\frac{1083-612\sqrt{2}}{217}$\\ \hline
\eqref{cw7}&$(  \frac{1}{1300},\frac{900}{13})$&$\frac{65\sqrt{217}}{1302}$  & $\frac{1103\mp7800\sqrt{2}}{141267}$ &  $\frac{\sqrt{-78}}{6},\frac{\sqrt{-78}}{2}$  &$\frac{-5983\mp2097\sqrt{2}}{83720}$ & $\frac{-4420\sqrt{3}}{423801}+\frac{168740\sqrt{6}}{141267}, \frac{4420\sqrt{3}}{141267}+\frac{168740\sqrt{6}}{47089}$ & $\frac{1083-612\sqrt{2}}{217}, \frac{361}{651}\hspace{-2pt}+\hspace{-2pt}\frac{68\sqrt{2}}{217}$\\ \hline
\eqref{cw8}&$(  \frac{-1}{5776}\hspace{-1pt},\hspace{-2pt}\frac{-83521}{361})$& $\frac{722\sqrt{329}}{16779}$  & $\frac{208329\pm25840\sqrt{65}}{974169}$ & $\frac{\sqrt{-130}}{2}\hspace{-1pt},\hspace{-2pt}\frac{\sqrt{-130}}{10}$  &$\frac{-9887}{123830}\hspace{-2pt}\pm\hspace{-2pt}\frac{5184\sqrt{65}}{804895}$ &
$\frac{691220\sqrt{26}}{324723}
\hspace{-2pt}+\hspace{-2pt}\frac{168740\sqrt{10}}{974169}
\hspace{-1pt},\hspace{-1pt}
\frac{138244\sqrt{26}}{324723}
\hspace{-2pt}-\hspace{-2pt}\frac{33748\sqrt{10}}{974169}$
&
$\frac{841+96\sqrt{65}}{1645}\hspace{-1pt},\hspace{-2pt}
\frac{4205-480\sqrt{65}}{329}$
  \\ \hline
\eqref{cw9}&$(  \frac{1}{5780},\frac{1156}{5})$& $\frac{85\sqrt{329}}{1974}$ & $\frac{208329\mp25840\sqrt{65}}{974169}$ &  $\frac{\sqrt{-130}}{10}\hspace{-1pt},\hspace{-2pt}\frac{\sqrt{-130}}{2}$  &$\frac{-9887}{123830}\hspace{-2pt}\mp\hspace{-2pt}\frac{5184\sqrt{65}}{804895}$ &
$\frac{138244\sqrt{26}}{324723}
\hspace{-2pt}-\hspace{-2pt}\frac{33748\sqrt{10}}{974169}
\hspace{-1pt},\hspace{-1pt}
\frac{691220\sqrt{26}}{324723}
\hspace{-2pt}+\hspace{-2pt}\frac{168740\sqrt{10}}{974169}$
& $\frac{4205-480\sqrt{65}}{329}\hspace{-1pt},\hspace{-2pt} \frac{841+96\sqrt{65}}{1645}$
  \\ \hline
\end{tabular}
\bigskip
\caption{CM points $\tau_\pm$ along with corresponding computations for cases \eqref{cw2}--\eqref{cw9}.}
\label{table3}
\textit{}
\end{table}
\end{center}
\normalsize
\endgroup
\vspace{-24pt}

\clearpage
\section{Remaining conjectures}

\begin{conjecture}[Z.-W. Sun \cite{S11}] \label{open}
% \begin{enumerate}
% \item
Define
\begin{align}
    \label{UT}
    U(t;x,b)=\sum_{n=0}^{\infty}\binom{2n}{n}\left(\sum_{k=0}^{n}\binom{n}{k}\binom{n+2k}{2k}\binom{2k}{k}x^{n-k}\right)(n+b)t^{n}
\end{align}
and
\begin{align}\label{VT}
    V(t;x,b)=\sum_{n=0}^{\infty}\binom{2n}{n}\left(\sum_{k=0}^{n}{\binom{2k}{k}^{\hspace{-2pt}2}\binom{2n-2k}{n-k}^{\hspace{-2pt}2}}{\binom{n}{k}^{\hspace{-2pt}-1}}x^{k}\right)(n+b)t^{n}.
\end{align}
Then
\begin{align}
    U\left(\frac{1}{2160};-324,\frac{103}{357}\right)&=\frac{30}{119\pi},\\
    U\left(\frac{1}{3645};486,0\right)&=\frac{10}{3\pi},\\
    U\left(\frac{-1}{1728};-324,\frac{1}{6}\right)&=\frac{12\sqrt{375+120\sqrt{10}}}{75\pi},\\
    U\left(\frac{-1}{160};-20,\frac{1}{4}\right)&=\frac{\sqrt{30}}{20\pi}\times\frac{5+\sqrt[3]{145+30\sqrt{6}}}{\sqrt[6]{145+30\sqrt{6}}},\\ %\ \ \textcolor{red}{\rm typo? It's fixed, thanks!}\\
    U\left(\frac{1}{27648};-2160, %%% 1290,
\frac{289}{1290}\right)&=\frac{16\sqrt{15}}{215\pi},\\
    U\left(\frac{1}{276480};12096,\frac{49}{804}\right)&=\frac{10\sqrt{15}}{67\pi},\\
    U\left(\frac{2}{135};-\frac{27}{8},\frac{5}{24}\right)&=\frac{5\sqrt{6}+4\sqrt{15}}{16\pi}, \\
   V\left(\frac{1}{576};5,\frac{5}{28}\right)&=\frac{9(2+\sqrt{2})}{28\pi},\\
    V\left(\frac{1}{576};-\frac{25}{16},\frac{31}{182}\right)&=\frac{189}{364\pi}.
\end{align}
    
% \end{enumerate}
\end{conjecture}
We did not find a formula for $U(t;x)$, but we can write $V(t;x)$ in terms of hyper-elliptic integrals, the same family of hyper-elliptic curves encountered in \cite{saga}. We listed an expression for $V(t;-\frac{25}{16})$ at \url{http://www.math.fsu.edu/~hoeij/files/Vtx}. The question is how to evaluate these hyper-elliptic integrals? So this part of Conjecture \ref{open} remains open as well.


\begin{thebibliography}{99}

\bibitem{AAR} G.E. Andrews, R. Askey, and R. Roy, \emph{Special Functions}, Cambridge University Press, Cambridge, 1999.

\bibitem{BBC}{N.D. Baruah, B.C. Berndt, and H.H. Chan, \emph{Ramanujan's series for $1/\pi$: a survey},} The American Mathematical Monthly \textbf{116} (7), 567--587 (2009).	


\bibitem{C17}{S. Cooper, \emph{Ramanujan's theta functions},} Springer, Cham (2017).

\bibitem{CWZ} S. Cooper, J. G. Wan, and W. Zudilin, \emph{Holonomic alchemy and series for $1/\pi$}, Analytic number theory, modular forms and $q$-hypergeometric series, 179--205, Springer Proc. Math. Stat. \textbf{221}, Springer, Cham, 2017.


\bibitem{vH}{L. van Hamme, \emph{Some conjectures concerning partial sums of generalized hypergeometric series},
In $p$-adic Functional Analysis (Nijmegen, 1996); Dekker: New York, NY, USA, 1997; Volume 192,
pp. 223–236.}


\bibitem{saga} M. van Hoeij, D. van Straten, and W. Zudilin, \emph{A hyperelliptic saga on a generating function of the squares of Legendre polynomials}, arXiv:2306.04921 (2023).

\bibitem{R14}{S. Ramanujan, \emph{Modular equations and approximations to $\pi$},} Quarterly Journal of Mathematics \textbf{45}, 350--372 (1914).

\bibitem{S11b}{Z.-W. Sun, \emph{List of conjectural series for powers of $\pi$ and other constants},} in: Ramanujan's Identities, Press of Harbin Institute of Tech., 2021, Chapter 5, pp. 205--261.

    \bibitem{S11}{Z.-W. Sun, \emph{New Conjectutres on Combinatorics and Number Theory}}, Harbin
Institute of Technology Press, Harbin, 2021.

\bibitem{S19}{Z.-W. Sun, \emph{Open conjectures on congruences},
Nanjing Univ. J. Math. Biquarterly 36 (2019), no. 1, 1--99.}

\bibitem{S23}{Z.-W. Sun, \emph{Some new series for $1/\pi$ motivated by congruences},  Colloq. Math. 173 (2023), no. 1, 89--109.}


\bibitem{S14}{Z.-W. Sun, \emph{Some new series for $1/\pi$
and related congruences},
Nanjing Univ. J. Math. Biquarterly 31(2014), no. 2, 150–164.}








\end{thebibliography}
\end{document}